\documentclass[12pt]{amsart}
 \usepackage[dvips]{epsfig}
 \usepackage{comment}
 \usepackage{enumerate}
 \usepackage{amsgen, amstext,amsbsy,amsopn, amsthm, amsfonts,amssymb,amscd,amsmat
 h,euscript,enumerate,url,verbatim,calc,xypic, mathtools}

 \usepackage{latexsym}
 \usepackage{graphics}
 \usepackage{color}

\newcommand{\proset}{\,\mathrel{\lower 4pt\hbox{$\scriptscriptstyle/$}
\mkern -14mu\subseteq }\,} %for proper subset
 
 \newtheorem{theorem}{Theorem}[section]
  \newtheorem{corollary}[theorem]{Corollary}
 \newtheorem{lemma}[theorem]{Lemma}

\numberwithin{equation}{section}

\def\ker{\operatorname{ker}}
\def\coker{\operatorname{coker}}

\usepackage{amsmath}
 \makeatother 
 
\begin{document}
\date{\today}

\title{Projective Bundle formula for Heller's relative $K_{0}$}
 \author{Vivek Sadhu} 
 
 \address{Department of Mathematics, Indian Institute of Science Education and Research Bhopal, Bhopal Bypass Road, Bhauri, Bhopal-462066, Madhya Pradesh, India}
 \email{ vsadhu@iiserb.ac.in, viveksadhu@gmail.com}
 \thanks{Author was supported by SERB-DST MATRICS grant MTR/2018/000283}

 \begin{abstract}
  In this article, we study the Heller relative $K_{0}$ group of the map $\mathbb{P}_{X}^{r} \to \mathbb{P}_{S}^{r},$ where $X$ and $S$ are quasi-projective schemes over a commutative ring. More precisely, we prove that the projective bundle formula holds for Heller's relative $K_{0},$ provided $X$ is flat over $S.$ As a corollary, we get a description of the relative group $K_{0}(\mathbb{P}_{X}^{r} \to \mathbb{P}_{S}^{r})$ in terms of generators and relations, provided  $X$ is affine and flat over $S.$
  \end{abstract}

\maketitle
 
\section{Introduction}
Throughout this article, we shall assume that all schemes have an ample family of line bundles. For example, every quasi-projective scheme over a commutative ring has an ample family of line bundles. Usually, $K$-theory of vector bundles behaves well under this assumption. By a quasi-projective scheme will always mean a scheme which is quasi-projective over some commutative ring.

Let $\mathcal{E}$ be a vector bundle of rank $r+1$ over a quasi-projective scheme $X,$ and $\mathbb{P}(\mathcal{E})$ denote the associated projective space bundle with the structure map $\pi: \mathbb{P}(\mathcal{E}) \to X.$ Then the projective bundle theorem for algebraic $K$-theory says that (see Theorem V.1.5 of \cite{wei 1}) for all $q \in \mathbb{Z},$ there is an isomorphism of groups
$$ K_{q}(X)^{r+1} \stackrel{\cong}\to K_{q}(\mathbb{P}(\mathcal{E})).$$ The goal of this article is to present a relative version of the above stated result for Heller's relative $K_{0}.$ 

Given a map of schemes $f: X \to S,$ let $K(f)$ denote the homotopy fibre of $K(S) \to K(X).$ Here $K(X)$ denotes the non-connective Bass $K$-theory spectrum of the scheme $X.$ Then $K_{n}(f),$ the $n$-th relative $K$-group of $f,$ is defined as $\pi_{n}K(f).$

In \cite{Hel}, Heller defined the relative $K_{0}$-group for a functor between certain categories. In this article, we are interested in Heller's relative $K_{0}$-groups for the pullback functor between the categories of vector bundles associated with a map of schemes. Given a map of schemes $f: X \to S,$ Heller's relative $K_{0}$ of $f^{*}$ is generated by the triples $(V_{1}, \alpha, V_{2})$ with suitable relations, where $V_{1}, V_{2}$ are vector bundles over $S$ and $\alpha: f^{*}V_{1}\to f^{*}V_{2}$ is an isomorphism (see Section \ref{k group definition}). We write $K_{0}^{He}(f)$ for the Heller relative $K_{0}$ of $f^{*}.$

It is natural to wonder the following: Is there any group isomorphism between $K_{0}(f)$ and $K_{0}^{He}(f)$? It is known that there is a group isomorphism if  $f: X \to S$ is a map of schemes with $X$ affine. More precisely, if $f: X \to S$ is a map of schemes with $X$ affine then there is an isomorphism of groups $K_{0}^{He}(f) \stackrel{\cong}\to K_{0}(f)$(see  \cite [Theorem 1.5 and Example 1.16]{RI}).
%\end{enumerate}
For an arbitrary $X,$ we do not know whether the above mentioned result is true or not. Therefore, it is interesting to study the group $K_{0}^{He}(\mathbb{P}_{X}^{r} \to \mathbb{P}_{S}^{r}).$

Given a map of quasi-projective schemes $f: X \to S,$ we have the following commutative diagram
\begin{equation}\label{basicdiagram1}
 \begin{CD}
   \mathbb{P}(\mathcal{E})\times_{S} X  =\mathbb{P}(f^{*}\mathcal{E}) @>  \mathbb{P}(f) >> \mathbb{P}(\mathcal{E})\\
    @V \pi_{X} VV  @V \pi_{S} VV \\                 
    X @> f >> S,
\end{CD}
\end{equation} where $\mathcal{E}$ is a vector bundle over $S.$ Note that if $\mathcal{E}= \mathcal{O}_{S}^{r+1}$ then $\mathbb{P}(f)$ is just the map  $\mathbb{P}_{X}^{r} \to \mathbb{P}_{S}^{r}.$ Under this situation, it is not hard to see that there is a group isomorphism (see Lemma \ref{projective bundle formula for K_n}) $$K_{n}(f)^{r+1} \stackrel{\cong}\to K_{n}(\mathbb{P}(f)).$$

We prove a similar result for relative $K_{0}^{He}$ in section \ref{first main theorem}. More precisely, we prove the following theorem (see Theorem \ref{proj formula for k heller}).

\begin{theorem}\label{projective bundle for relative K_0}
Let  $\mathbb{P}(f): \mathbb{P}(f^{*}\mathcal{E}) \to \mathbb{P}(\mathcal{E})$ be a map as in diagram (\ref{basicdiagram1}). Assume that $f: X \to S$ is a flat map of quasi-projective schemes and ${\rm {rank}}(\mathcal{\mathcal{E}})= r+1.$  Then there is an isomorphism of groups
 $$K_{0}^{He}(f)^{r+1} \stackrel{\cong}\to K_{0}^{He}(\mathbb{P}(f)).$$
\end{theorem}

Using the Theorem \ref{projective bundle for relative K_0}, we also prove the following (see Corollary \ref{He=k}):
\begin{corollary}
 Suppose $\mathbb{P}(f),$ $f$ and $\mathcal{E}$ are as in Theorem \ref{projective bundle for relative K_0}.  Further, we assume that $X$ is an affine scheme. Then there is an isomorphism of groups $K_{0}^{He}(\mathbb{P}(f))\cong K_{0}(\mathbb{P}(f)).$
\end{corollary}

{\bf Acknowledgements:} The author would like to thank Charles Weibel for his helpful comments during the preparation of this article. He would also like to thank the referee for valuable comments and suggestions.

 \section{Preliminaries}
  
  For a scheme $X,$ let ${\bf{Mod}}(X)$ denote the category of $\mathcal{O}_{X}$-modules and ${\bf{VB}}(X)$ denote the category of vector bundles on $X$. It is well known that  ${\bf{Mod}}(X)$ is an abelian category and ${\bf{VB}}(X)$ is an exact subcategory of ${\bf{Mod}}(X).$
  
  Given a map of schemes $f: X \to S,$ we define a category  ${\bf{Mod}}(f)$ whose objects are triples $(\mathcal{F}_{1}, \alpha, \mathcal{F}_{2})$ with $\mathcal{F}_{1}, \mathcal{F}_{2} \in {\bf{Mod}}(S)$ and $\alpha: f^{*}\mathcal{F}_{1}\cong f^{*}\mathcal{F}_{2}$ in ${\bf{Mod}}(X).$ A morphism $(\mathcal{F}_{1}, \alpha, \mathcal{F}_{2}) \to (\mathcal{F}_{1}^{'}, \alpha^{'}, \mathcal{F}_{2}^{'})$ is a pair of maps $u: \mathcal{F}_{1} \to \mathcal{F}_{1}^{'}$, $v: \mathcal{F}_{2} \to \mathcal{F}_{2}^{'}$ in ${\bf{Mod}}(S)$ such that $\alpha^{'} f^{*}u= f^{*}v~ \alpha$. In a similar way, we can define ${\bf{VB}}(f)$ by replacing $\mathcal{O}_{S}$-modules with vector bundles on $S$. 
  
  If $f$ is a flat map then $(\ker(u), \tilde{\alpha}, \ker(v))\in {\bf{Mod}}(f),$ where $$\tilde{\alpha}: f^{*}\ker(u)\cong \ker(f^{*}u)\cong \ker(f^{*}v)\cong f^{*}\ker(v).$$
  
  Hereafter, we assume that $f: X \to S$ is a flat map of schemes.
  
  Given a morphism $(u, v): (\mathcal{F}_{1}, \alpha, \mathcal{F}_{2}) \to (\mathcal{F}_{1}^{'}, \alpha^{'}, \mathcal{F}_{2}^{'})$ in ${\bf{Mod}}(f),$ we define
  $$\ker(u, v):= (\ker(u), \tilde{\alpha}, \ker(v)) ~{\rm {and}}~ \coker (u, v):= (\coker(u), \tilde{\alpha{'}}, \coker(v)),$$ where $\tilde{\alpha{'}}$ induces from $\alpha^{'}.$ Under this definition, ${\bf{Mod}}(f)$ is an abelian category.

  The following lemma is an easy observation.
  \begin{lemma}\label{exact seq}
  Let 
  \begin{equation}\label{exactness in Mod}
   (\mathcal{F}_{1}, \alpha, \mathcal{F}_{2}) \stackrel{(u,v)}\to (\mathcal{F}_{1}^{'}, \alpha^{'}, \mathcal{F}_{2}^{'})\stackrel{(u^{'}, v^{'})}\to (\mathcal{F}_{1}^{''}, \alpha^{''}, \mathcal{F}_{2}^{''})
  \end{equation} be a sequence in ${\bf{Mod}}(f).$ Then (\ref{exactness in Mod}) is exact if and only if 
   $$ \mathcal{F}_{1} \stackrel{u}\to \mathcal{F}_{1}^{'} \stackrel{u^{'}}\to \mathcal{F}_{1}^{''}$$
   and $$\mathcal{F}_{2} \stackrel{v}\to \mathcal{F}_{2}^{'} \stackrel{v^{'}}\to \mathcal{F}_{2}^{''}$$ both are exact in ${\bf{Mod}}(S).$
  \end{lemma}

  Suppose that the sequence
   $$0 \to (\mathcal{F}_{1}, \alpha, \mathcal{F}_{2}) \stackrel{(u,v)}\to (\mathcal{F}_{1}^{'}, \alpha^{'}, \mathcal{F}_{2}^{'})\stackrel{(u^{'}, v^{'})}\to (\mathcal{F}_{1}^{''}, \alpha^{''}, \mathcal{F}_{2}^{''})\to 0$$ is exact in ${\bf{Mod}}(f)$ with $(\mathcal{F}_{1}, \alpha, \mathcal{F}_{2}), (\mathcal{F}_{1}^{''}, \alpha^{''}, \mathcal{F}_{2}^{''}) \in {\bf{VB}}(f).$ Since ${\bf{VB}}(S)$ is an exact subcategory of ${\bf{Mod}}(S),$ $p: \mathcal{F}_{1}^{'}\cong V_{1}$ and $q: \mathcal{F}_{2}^{'}\cong V_{2},$ where $V_{1}, V_{2} \in {\bf{VB}}(S)$(using Lemma \ref{exact seq}). Then $(p, q): (\mathcal{F}_{1}^{'}, \alpha^{'}, \mathcal{F}_{2}^{'})\cong (V_{1}, \beta, V_{2})$ in ${\bf{VB}}(f),$ where $\beta= f^{*}q.\alpha^{'}.f^{*}p^{-1}.$ This proves the following:
   
   \begin{lemma}\label{exact subcategory}
    ${\bf{VB}}(f)$ is an exact subcategory of ${\bf{Mod}}(f).$
   \end{lemma}

\section{Relative K-theory}\label{k group definition}
Let $f: X \to S$ be a map of schemes. Let $K(f)$ be the homotopy fibre of $K(S) \to K(X).$ Here $K(X)$ denotes the non-connective Bass $K$-theory spectrum of the scheme $X.$ Then $$K_{n}(f):= \pi_{n}K(f)$$ for $n\in \mathbb{Z}.$

In \cite{Hel}, Heller introduced the relative $K_{0}$-groups for a functor between certain categories (see also \cite{Bass-Tata}). Following Heller \cite{Hel}, very recently R. Iwasa in \cite{RI} define the relative $K_{0}$-groups for an exact functor between small exact categories. We now recall the definition from \cite{RI} in a special situation. For more details, we refer to section 1 of \cite{RI}. Consider ${\bf{VB}}(f)$ as the relative category associated to the pullback functor $f^{*}: {\bf{VB}}(S) \to {\bf{VB}}(X).$ We define $K_{0}(f^{*})$ to be the abelian group generated by $[(V_{1}, \alpha, V_{2})],$ where $(V_{1}, \alpha, V_{2}) \in {\bf{VB}}(f).$ The relations are:
\begin{itemize}
\item  $[V^{'}] + [V^{''}] = [V]$ for every exact sequence $0 \to V^{'} \to V \to V^{''}\to 0$ in  ${\bf{VB}}(f);$
\item  $[(V_{1},\alpha,V_{2})] + [(V_{2},\beta,V_{3})] = [(V_{1},\beta\alpha,V_{3})]$ for every pair $(V_{1},\alpha,V_{2}), (V_{2},\beta,V_{3})$ of objects in ${\bf{VB}}(f).$
\end{itemize}

We prefer to write $K_{0}^{He}(f)$ for $K_{0}(f^{*}).$

Next, we discuss the relationship between $K_{0}(f)$ and $K_{0}^{He}(f).$ By applying Theorem 1.5 of \cite{RI} to the functor $f^{*}: {\bf{VB}}(S) \to {\bf{VB}}(X)$ with $X$ affine, we get the following: 
\begin{lemma}\label{K^b= K for affine}
 Let $f: X \to S$ be a map of schemes with $X$ affine. Then there is an isomorphism $K_{0}^{He}(f)\stackrel{\cong}\to K_{0}(f).$
\end{lemma}

The next result is the projective bundle formula for relative $K$-theory.
\begin{lemma}\label{projective bundle formula for K_n}
  Let  $\mathbb{P}(f): \mathbb{P}(f^{*}\mathcal{E}) \to \mathbb{P}(\mathcal{E})$ be a map as in diagram (\ref{basicdiagram1}). Assume that the rank of $\mathcal{E}$ is $r+1.$ Then there is a natural isomorphism of groups $K_{n}(f)^{r+1}\stackrel{\cong}\to K_{n}(\mathbb{P}(f))$ for all $n \in \mathbb{Z}.$
 \end{lemma}
 
 \begin{proof}
  By Theorem V.1.5 of \cite{wei 1}, we have an equivalence $K(S)^{r+1}\simeq K(\mathbb{P}(\mathcal{E})).$ Consider the following commutative diagram of $K$-theory spectra
  $$\begin{CD}
   K(f)^{r+1} @>>> K(S)^{r+1} @>>> K(X)^{r+1}\\
   @VVV       @VV\simeq V          @VV\simeq V \\
   K(\mathbb{P}(f)) @>>> K(\mathbb{P}(\mathcal{E})) @>>>  K(\mathbb{P}(f^{*}\mathcal{E})).
  \end{CD}$$
 
This induces a commutative diagram of long exact sequences
\small
$$\begin{CD}
 \dots @>>> K_{n+1}(X)^{r+1} @>>> K_{n}(f)^{r+1} @>>> K_{n}(S)^{r+1} @>>> K_{n}(X)^{r+1} @>>> \dots \\
 @.   @VV\cong V    @VVV      @VV\cong V       @VV\cong V    @.     \\
 \dots @>>> K_{n+1}(\mathbb{P}(f^{*}\mathcal{E})) @>>> K_{n}(\mathbb{P}(f)) @>>> K_{n}(\mathbb{P}(\mathcal{E})) @>>> K_{n}(\mathbb{P}(f^{*}\mathcal{E})) @>>> \dots.
\end{CD}$$
\small  Hence the assertion.
 \end{proof}
The rest of the paper is dedicated to proving the projective bundle formula for $K_{0}^{He}.$

\section{Mumford-regular bundles}
Let $\mathcal{E}$ be a vector bundle of finite rank over a quasi-projective scheme $X.$ Let $\mathbb{P}= \mathbb{P}(\mathcal{E})$ be the associated projective space bundle. There is a natural map $\pi=\pi_{X}: \mathbb{P}(\mathcal{E}) \to X.$ Some relevant details can be found in \cite[Chapter  8]{GT}.

A quasi-coherent $\mathcal{O}_{\mathbb{P}}$-module $\mathcal{F}$ is said to be { \it Mumford-regular} if for all $q>0$ the higher derived sheaves $R^{q}\pi_{*}(\mathcal{F}(-q))=0.$ Here $\mathcal{F}(n)$ is the twisted sheaf $\mathcal{F}\otimes \mathcal{O}_{\mathbb{P}}(n).$ 

 We now recall some known results pertaining to Mumford-regular modules. For details, we refer to \cite[Section 8]{DQ} and \cite[Chapter II.8]{wei 1}.

\begin{lemma}\label{besic1}
 If $\mathcal{F}$ is Mumford-regular, then:
 \begin{enumerate}
  \item The twist $\mathcal{F}(n)$ are Mumford-regular for all $n\geq 0.$
  \item The canonical map $\varepsilon: \pi^{*}\pi_{*}(\mathcal{F}) \to \mathcal{F}$ is onto.
 \end{enumerate}
\end{lemma}

\begin{proof}
 See Proposition II.8.7.3 of \cite{wei 1}. \end{proof}

\begin{lemma}\label{basic3}
 The functor $\pi_{*}$ is exact from Mumford-regular modules to $\mathcal{O}_{X}$-modules.
\end{lemma}
\begin{proof}
 See Lemma II.8.7.4 of \cite{wei 1}. \end{proof}

\begin{lemma}\label{basic2}
Let $\mathcal{F}$ be a vector bundle on $\mathbb{P}.$
 \begin{enumerate}
  \item $\mathcal{F}(n)$ is a Mumford-regular vector bundle on $\mathbb{P}$ for all large enough $n.$
  \item If $\mathcal{F}$ is Mumford-regular, then $\pi_{*}\mathcal{F}$ is a vector bundle on $X.$
 \end{enumerate}
\end{lemma}

\begin{proof}
 See Lemma II.8.7.5 of \cite{wei 1}. \end{proof}

Let $f: X \to S$ be a map of quasi-projective schemes and $\mathcal{E}$ be a vector bundle of finite rank on $S.$ Then we have a commutative diagram (\ref{basicdiagram1}). Let ${\bf {MR}(\mathbb{P}(\mathcal{E}))}$ denote the category of Mumford-regular vector bundles. Now, we define a category  ${\bf{MR}}(\mathbb{P}(f))$ whose objects are triples $(\mathcal{F}_{1}, \alpha, \mathcal{F}_{2})$ with $\mathcal{F}_{1}, \mathcal{F}_{2} \in {\bf {MR}(\mathbb{P}(\mathcal{E}))}$ and $\alpha: \mathbb{P}(f)^{*}\mathcal{F}_{1}\cong \mathbb{P}(f)^{*}\mathcal{F}_{2}$ in ${\bf {MR}(\mathbb{P}(\mathcal{E}))}.$ A morphism $(\mathcal{F}_{1}, \alpha, \mathcal{F}_{2}) \to (\mathcal{F}_{1}^{'}, \alpha^{'}, \mathcal{F}_{2}^{'})$ is a pair of maps $u: \mathcal{F}_{1} \to \mathcal{F}_{1}^{'}$, $v: \mathcal{F}_{2} \to \mathcal{F}_{2}^{'}$ in ${\bf {MR}(\mathbb{P}(\mathcal{E}))}$ such that $\alpha^{'} \mathbb{P}(f)^{*}u= \mathbb{P}(f)^{*}v~ \alpha$.

Assume that $f: X \to S$ is a flat map. Let $(\mathcal{F}_{1}, \alpha, \mathcal{F}_{2} )\in {\bf {MR}}(\mathbb{P}(f)).$ By Lemma \ref{besic1}(2), there are canonical onto  maps $\varepsilon_{i}: \pi_{S}^{*}\pi_{S*}(\mathcal{F}_{i}) \to \mathcal{F}_{i}$ for $i=1, 2.$ %Note that $\varepsilon_{i}$'s are natural in $\mathcal{F}_{i}.$ 
We also have
$$\mathbb{P}(f)^{*}\pi_{S}^{*}\pi_{S*}(\mathcal{F}_{i})= \pi_{X}^{*}f^{*}\pi_{S*}\mathcal{F}_{i}\cong \pi_{X}^{*}\pi_{X*}\mathbb{P}(f)^{*}\mathcal{F}_{i} ~~{\rm {for}}~~ i=1,2,$$ where the first equality by the commutativity of the diagram (\ref{basicdiagram1}) and the second isomorphism by the flat base change theorem (see Lemma 30.5.2 of \cite{sp}). Thus, we get an isomorphism $\mathbb{P}(f)^{*}\pi_{S}^{*}\pi_{S*}(\mathcal{F}_{1})\cong \mathbb{P}(f)^{*}\pi_{S}^{*}\pi_{S*}(\mathcal{F}_{2})$ and is denoted by $\pi_{X}^{*}\pi_{X*}\alpha.$ Since $\pi_{S}$ is quasi-compact and separated, $\pi_{S}^{*}\pi_{S*}(\mathcal{F}_{i})\in {\bf {MR}(\mathbb{P}(\mathcal{E}))}$  for $i= 1, 2$ by Example II.8.7.2 of \cite{wei 1}. This implies that $$(\pi_{S}^{*}\pi_{S*}\mathcal{F}_{1}, \pi_{X}^{*}\pi_{X*}\alpha, \pi_{S}^{*}\pi_{S*}\mathcal{F}_{2}) \in {\bf{MR}}(\mathbb{P}(f)).$$ Note that the canonical map $\varepsilon: \pi^{*}\pi_{*}(\mathcal{F}) \to \mathcal{F}$ is natural in $\mathcal{F}.$ So, the diagram 
\begin{equation}\label{basicdiagram}
 \begin{CD}
     \mathbb{P}(f)^{*}\pi_{S}^{*}\pi_{S*}\mathcal{F}_{1}\cong \pi_{X}^{*}\pi_{X*}\mathbb{P}(f)^{*}\mathcal{F}_{1} @> \mathbb{P}(f)^{*}\varepsilon_{1} >> \mathbb{P}(f)^{*}\mathcal{F}_{1}\\
    @V \pi_{X}^{*}\pi_{X*}\alpha VV  @V \alpha VV \\                 
  \mathbb{P}(f)^{*}\pi_{S}^{*}\pi_{S*}\mathcal{F}_{2}\cong \pi_{X}^{*}\pi_{X*}\mathbb{P}(f)^{*}\mathcal{F}_{2} @> \mathbb{P}(f)^{*}\varepsilon_{2} >> \mathbb{P}(f)^{*}\mathcal{F}_{2}
\end{CD}
\end{equation} is commutative by the naturality of $\mathbb{P}(f)^{*}\varepsilon$.  Hence we get the following (by Lemma \ref{exact seq}):

\begin{lemma}\label{surjection}
  $(\varepsilon_{1}, \varepsilon_{2}): (\pi_{S}^{*}\pi_{S*}\mathcal{F}_{1}, \pi_{X}^{*}\pi_{X*}\alpha, \pi_{S}^{*}\pi_{S*}\mathcal{F}_{2}) \to (\mathcal{F}_{1}, \alpha, \mathcal{F}_{2})$ is a morphism in  ${\bf{MR}}(\mathbb{P}(f))$ and it is onto. Here $f: X \to S$ is a flat map.
\end{lemma}

\section{Relative version of Quillen's Resolution Theorem}
In this section, we prove a relative version of Quillen's resolution theorem which will play an important role in the later part of the paper. Throughout this section, $f: X \to S$ is a flat map of quasi-projective schemes.

First, we recall some notations from \cite{DQ} and \cite{wei 1}. Given a Mumford-regular $\mathcal{O}_{\mathbb{P}}$-module $\mathcal{F},$ we define a sequence of $\mathcal{O}_{X}$-modules $T_{n}= T_{n}\mathcal{F}$ and $\mathcal{O}_{\mathbb{P}}$-modules $Z_{n}= Z_{n}\mathcal{F}$ as follows. Starting with $T_{0}\mathcal{F}=\pi_{*}\mathcal{F}$ and $Z_{-1}= \mathcal{F}.$ Since $\mathcal{F}$ is Mumford-regular, there is a canonical onto map  $\varepsilon: \pi^{*}\pi_{*}(\mathcal{F}) \to \mathcal{F}$ (see Lemma \ref{besic1}). Let $Z_{0}\mathcal{F}= \ker \varepsilon.$ So, we get an exact sequence
$$ 0 \to Z_{0}\mathcal{F} \to \pi^{*}T_{0}\mathcal{F} \to Z_{-1}\mathcal{F}\to 0.$$ Inductively, we define 
$$ T_{n}\mathcal{F}= \pi_{*}Z_{n-1}(n),~  Z_{n}\mathcal{F}= \ker(\varepsilon)(-n),$$ where $\varepsilon$  is the canonical map $\pi^{*}\pi_{*}Z_{n-1}(n) \to Z_{n-1}(n).$ Therefore, we have sequences
\begin{equation}\label{Resolution seq}
 0 \to Z_{n}(n) \to \pi^{*}T_{n}\mathcal{F} \to Z_{n-1}(n)\to 0,
\end{equation}
which are exact except possibly at $Z_{n-1}(n).$ Now we state a result which is known as the Quillen Resolution theorem.

\begin{theorem}\label{QRT}
 Let $\mathcal{F}$ be a vector bundle on $\mathbb{P}(\mathcal{E}),$ rank$(\mathcal{E})= r+1.$ If $\mathcal{F}$ is Mumford-regular, then $Z_{r}=0,$ and the sequences (\ref{Resolution seq}) are exact for $n\geq 0,$ so there is an exact sequence 
 \begin{equation}
  0 \to (\pi^{*}T_{r}\mathcal{F})(-r) \stackrel{\varepsilon(-r)}\to \dots \to (\pi^{*}T_{i}\mathcal{F})(-i)\stackrel{\varepsilon(-i)}\to \dots \stackrel{\varepsilon(-1)}\to \pi^{*}T_{0}\mathcal{F} \to Z_{-1}\mathcal{F}\to 0.
 \end{equation}
\end{theorem}

\begin{proof}
 See Theorem 8.7.8 of \cite{wei 1}.
\end{proof}

Next, our goal is to prove a relative version of the above theorem. To do this let us first fix some notation.

For $n \in \mathbb{Z},$ we define the relative twist functor $(n)^{rel}: {\bf{Mod}}(\mathbb{P}(f)) \to {\bf{Mod}}(\mathbb{P}(f))$ by $$(n)^{rel}(\mathcal{F}_{1}, \alpha, \mathcal{F}_{2})= (\mathcal{F}_{1}, \alpha, \mathcal{F}_{2})(n):= (\mathcal{F}_{1}(n), \alpha(n), \mathcal{F}_{2}(n)),$$ where  $\mathbb{P}(f)$ as in diagram (\ref{basicdiagram1}) and $\alpha(n):= \alpha \otimes id: \mathbb{P}(f)^{*}\mathcal{F}_{1}\otimes \mathcal{O}_{\mathbb{P}(f^{*}\mathcal{E})}(n)\cong \mathbb{P}(f)^{*}\mathcal{F}_{2}\otimes \mathcal{O}_{\mathbb{P}(f^{*}\mathcal{E})}(n). $
   
   \begin{lemma}\label{shift exact}
    For $n \in \mathbb{Z},$ $(n)^{rel}$ is an exact functor on ${\bf{Mod}}(\mathbb{P}(f)).$
   \end{lemma}
   
   \begin{proof}
    Since $\mathcal{O}_{\mathbb{P}(\mathcal{E})}(n)$ is flat over $\mathcal{O}_{\mathbb{P}(\mathcal{E})},$  the twist $(n)$ is an exact functor on ${\bf{Mod}}(\mathbb{P}(\mathcal{E})).$ Hence the result follows by Lemma \ref{exact seq}.
   \end{proof}

Let $\mathcal{F}:= (\mathcal{F}_{1}, \alpha, \mathcal{F}_{2}) \in {\bf{MR}}(\mathbb{P}(f)).$  Since $f$ is a flat map, 
$$\pi_{X*}\alpha: f^{*}\pi_{S*}\mathcal{F}_{1}\cong \pi_{X*}\mathbb{P}(f)^{*}\mathcal{F}_{1}\cong \pi_{X*}\mathbb{P}(f)^{*}\mathcal{F}_{2}\cong f^{*}\pi_{S*}\mathcal{F}_{2}.$$

Let $$ T_{0}\alpha= \pi_{X*}\alpha, ~~ Z_{-1}\alpha= \alpha.$$ Then we define $$\mathcal{T}_{0}(\mathcal{F})= (T_{0}\mathcal{F}_{1}, T_{0}\alpha, T_{0}\mathcal{F}_{2}) ~~{\rm and}~~ \mathcal{Z}_{-1}\mathcal{F}=(Z_{-1}\mathcal{F}_{1}, Z_{-1}\alpha, Z_{-1}\mathcal{F}_{2})= \mathcal{F},$$ where $T_{0}\mathcal{F}_{i}= \pi_{S*}\mathcal{F}_{i}$ for $i=1, 2.$  Clearly, $\mathcal{T}_{0}(\mathcal{F})\in {\bf{Mod}}(f).$  Let $\mathcal{Z}_{0}\mathcal{F}= (Z_{0}\mathcal{F}_{1}, Z_{0}\alpha, Z_{0}\mathcal{F}_{2}),$ where $Z_{0}\alpha= \pi_{X}^{*}T_{0}\alpha.$ Since $\mathbb{P}(f)$ is a flat map, $\mathcal{Z}_{0}\mathcal{F}\in {\bf{Mod}}(\mathbb{P}(f)).$  Inductively, we define

$$\mathcal{T}_{n}(\mathcal{F})= (T_{n}\mathcal{F}_{1}, T_{n}\alpha, T_{n}\mathcal{F}_{2}) ~~{\rm and}~~ \mathcal{Z}_{n}\mathcal{F}=(Z_{n}\mathcal{F}_{1}, Z_{n}\alpha, Z_{n}\mathcal{F}_{2}),$$ where $T_{n}\alpha= \pi_{X*}((Z_{n-1}\alpha)(n))$ and $Z_{n}\alpha=(\pi_{X}^{*}T_{n}\alpha)(-n).$ One can easily check that for each $n \in \mathbb{N},$ $\mathcal{T}_{n}(\mathcal{F}) \in {\bf{Mod}}(f)$ and $\mathcal{Z}_{n}\mathcal{F} \in {\bf{Mod}}(\mathbb{P}(f)).$ Define $$\pi_{S}^{*}\mathcal{T}_{n}(\mathcal{F}):= (\pi_{S}^{*}T_{n}\mathcal{F}_{1}, \pi_{X}^{*}T_{n}\alpha, \pi_{S}^{*}T_{n}\mathcal{F}_{2}).$$ Thus we have sequences 

\begin{equation}\label{relative Resolution seq}
 0 \to \mathcal{Z}_{n}\mathcal{F}(n) \to \pi_{S}^{*}\mathcal{T}_{n}\mathcal{F} \stackrel{(\varepsilon_{1}, \varepsilon_{2})}\longrightarrow \mathcal{Z}_{n-1}\mathcal{F}(n)
\end{equation} in ${\bf{Mod}}(\mathbb{P}(f)).$ We are now ready to prove a relative version of Theorem \ref{QRT}.

\begin{theorem}\label{relative Q Resolution}
 Let $\mathcal{F}:= (\mathcal{F}_{1}, \alpha, \mathcal{F}_{2}) \in {\bf{MR}}(\mathbb{P}(f)),$ where  $\mathbb{P}(f)$ as in diagram (\ref{basicdiagram1}) with rank$(\mathcal{E})= r+1.$ Then there is an exact sequence
 \tiny
 \begin{equation}\label{relative exact Q Resolution}
  0 \to (\pi_{S}^{*}\mathcal{T}_{r}\mathcal{F})(-r) \stackrel{(\varepsilon_{1}(-r),\varepsilon_{2}(-r))} \longrightarrow \dots \to (\pi_{S}^{*}\mathcal{T}_{i}\mathcal{F})(-i)\stackrel{(\varepsilon_{1}(-i),\varepsilon_{2}(-i))}\longrightarrow \dots \stackrel{(\varepsilon_{1}(-1),\varepsilon_{2}(-1))}\longrightarrow \pi_{S}^{*}\mathcal{T}_{0}\mathcal{F} \to \mathcal{F}\to 0
 \end{equation} 
 \tiny
 \normalsize
 in ${\bf{Mod}}(\mathbb{P}(f)).$
 Moreover, each $\mathcal{F}\mapsto \mathcal{T}_{i}\mathcal{F}$ is an exact functor from ${\bf{MR}}(\mathbb{P}(f))$ to ${\bf{VB}}(f).$
\end{theorem}
\begin{proof}
 Since $\mathcal{F}_{1}, \mathcal{F}_{2}$ both are Mumford-regular, so are $Z_{n-1}\mathcal{F}_{1}(n), Z_{n-1}\mathcal{F}_{2}(n)$ (see the proof of Theorem II. 8.7.8 of \cite{wei 1}). 
 By Lemma \ref{surjection}, the sequences (\ref{relative Resolution seq}) are exact at $\mathcal{Z}_{n-1}\mathcal{F}(n),$ i.e., we have exact sequences
 \begin{equation}\label{relative exact Q}
   0 \to \mathcal{Z}_{n}\mathcal{F}(n) \to \pi_{S}^{*}\mathcal{T}_{n}\mathcal{F} \stackrel{(\varepsilon_{1}, \varepsilon_{2})}\longrightarrow \mathcal{Z}_{n-1}\mathcal{F}(n)\to 0.
 \end{equation}
Now the twists of the sequences (\ref{relative exact Q}) fit together into the sequence of the form (\ref{relative exact Q Resolution}).

For the second part, each $\mathcal{F}\mapsto T_{i}\mathcal{F}$ is an exact functor from ${\bf {MR}(\mathbb{P}(\mathcal{E}))}$ to ${\bf{VB}}(X)$ by Corollary II.8.7.9 of \cite{wei 1}. Hence the assertion follows from Lemma \ref{exact seq}. \end{proof}

\section{Projective bundle formula for relative $K_{0}^{He}$}\label{first main theorem}
In this section, we prove that the projective bundle formula holds for Heller's relative $K_{0}^{He}$ of a flat map. Throughout this section, $f: X \to S$ is a flat map of quasi-projective schemes. Also, $\mathbb{P}(f)$ always mean the map $\mathbb{P}(f^{*}\mathcal{E}) \to \mathbb{P}(\mathcal{E})$  as in  diagram (\ref{basicdiagram1}) with rank$(\mathcal{E})= r+1.$

   We observe in Lemma \ref{exact subcategory} that ${\bf{VB}}(f)$ is an exact subcategory of ${\bf{Mod}}(f).$ So we can define $K_{0}({\bf{VB}}(f))$ in the sense of Quillen absolute $K_{0}$ of exact categories. By definition, $K_{0}({\bf{VB}}(f))$ is the abelian group generated $[(V_{1}, \alpha, V_{2})],$ where $(V_{1}, \alpha, V_{2}) \in {\bf{VB}}(f),$ and  relations 
 $[V^{'}] + [V^{''}] = [V]$ for every exact sequence $0 \to V^{'} \to V \to V^{''}\to 0$ in  ${\bf{VB}}(f).$ It is denoted by $K_{0}^{Q}(f).$ Clearly, there is a natural surjection 
 \begin{equation}
  \eta^{f}: K_{0}^{Q}(f) \to K_{0}^{He}(f)
 \end{equation}
 and $\ker (\eta^{f})$ is generated by $[(V_{1},\alpha,V_{2})] + [(V_{2},\beta,V_{3})] - [(V_{1},\beta\alpha,V_{3})]$ for every pair $(V_{1},\alpha,V_{2}), (V_{2},\beta,V_{3})$ of objects in ${\bf{VB}}(f).$

The $n$-th twist of ${\bf{MR}}(\mathbb{P}(f)),$ notation ${\bf{MR}}(\mathbb{P}(f))(n),$ is a category consisting of objects $(\mathcal{F}_{1}, \alpha, \mathcal{F}_{2})$ of ${\bf{VB}}(\mathbb{P}(f))$ such that $(\mathcal{F}_{1}(-n), \alpha(-n), \mathcal{F}_{2}(-n))$ is in ${\bf{MR}}(\mathbb{P}(f)).$ Each ${\bf{MR}}(\mathbb{P}(f))(n)$ is an exact category because the relative twisting is an exact functor (see Lemma \ref{shift exact}) and the Mumford-regular modules are closed under extensions (see Lemma II.8.7.4 of \cite{wei 1}). By Lemma \ref{besic1}(1), we have
\small
\begin{equation*}
 {\bf{MR}}(\mathbb{P}(f))={\bf{MR}}(\mathbb{P}(f))(0)\subseteq {\bf{MR}}(\mathbb{P}(f))(-1)\subseteq \dots \subseteq {\bf{MR}}(\mathbb{P}(f))(n)\subseteq {\bf{MR}}(\mathbb{P}(f))(n-1)\subseteq \dots 
\end{equation*}
\small
\begin{theorem}\label{MR=VB}
 For all $n\leq 0,$ $K_{0}^{Q}{\bf{MR}}(\mathbb{P}(f))\cong K_{0}^{Q}{\bf{MR}}(\mathbb{P}(f))(n)\cong K_{0}^{Q}(\mathbb{P}(f))$ induced by the inclusion ${\bf{MR}}(\mathbb{P}(f))(n)\subset {\bf{VB}}(\mathbb{P}(f)).$
\end{theorem}

\begin{proof}
 Let $(\mathcal{F}_{1}, \alpha, \mathcal{F}_{2})\in {\bf{VB}}(\mathbb{P}(f)).$ By Lemma \ref{basic2}(1), $\mathcal{F}_{1}(n),$ $\mathcal{F}_{2}(n) \in {\bf {MR}(\mathbb{P}(\mathcal{E}))}$ for $n\geq 0$ large enough. Then $$(\mathcal{F}_{1}(n), \alpha(n), \mathcal{F}_{2}(n))\in {\bf{MR}}(\mathbb{P}(f))(-n)$$ for $n\geq 0$ large enough. So, it is clear that $\cup_{n\leq  0} {\bf{MR}}(\mathbb{P}(f))(n)= {\bf{VB}}(\mathbb{P}(f)).$ Since $K_{0}^{Q}$ commutes with filtered colimits (see Example II. 7.1.7 of \cite{wei 1}), we have $K_{0}^{Q}{\bf{VB}}(\mathbb{P}(f))= \varinjlim_{n} K_{0}^{Q}{\bf{MR}}(\mathbb{P}(f))(n).$ For each inclusion   $${\bf{MR}}(\mathbb{P}(f))(n)\subseteq {\bf{MR}}(\mathbb{P}(f))(n-1),$$ we have the induced map $l_{n}:K_{0}^{Q}{\bf{MR}}(\mathbb{P}(f))(n) \to  K_{0}^{Q}{\bf{MR}}(\mathbb{P}(f))(n-1).$ So, it is enough to show that each such $l_{n}$ is an isomorphism. For each $i>0,$ $$(\mathcal{F}_{1}, \alpha, \mathcal{F}_{2})\mapsto (\mathcal{F}_{1}(i)\otimes \pi_{S}^{*}\wedge^{i}\mathcal{E}, \alpha(i)\otimes id,  \mathcal{F}_{2}(i)\otimes \pi_{S}^{*}\wedge^{i}\mathcal{E})$$ defines an exact functor from ${\bf{MR}}(\mathbb{P}(f))(n-1)$ to ${\bf{MR}}(\mathbb{P}(f))(n).$ It induces a homomorphism $\lambda_{i}: K_{0}^{Q}{\bf{MR}}(\mathbb{P}(f))(n-1) \to K_{0}^{Q}{\bf{MR}}(\mathbb{P}(f))(n).$ For a vector bundle $\mathcal{F}$ in ${\bf{VB}}(\mathbb{P}(\mathcal{E})),$ we have the Koszul resolution (see the proof of Lemma 1.3 in the section 8 of \cite{DQ})
 $$ 0 \to \mathcal{F}\to \mathcal{F}(1)\otimes \pi^{*}\wedge \mathcal{E}^{\vee} \to \dots \to \mathcal{F}(r+1) \otimes \pi^{*}\wedge^{r+1}\mathcal{E}^{\vee} \to 0.$$ Here $\mathcal{E}^{\vee}$ denotes the dual of $\mathcal{E}.$ Similarly, a relative version of Koszul resolution is (by Lemma \ref{exact seq})
 \begin{multline*}
  0 \to (\mathcal{F}_{1}, \alpha, \mathcal{F}_{2}) \to (\mathcal{F}_{1}(1)\otimes \pi_{S}^{*}\wedge \mathcal{E}^{\vee}, \alpha(1)\otimes id,  \mathcal{F}_{2}(1)\otimes \pi_{S}^{*}\wedge \mathcal{E}^{\vee}) \\ \to \dots \to (\mathcal{F}_{1}(r+1)\otimes \pi_{S}^{*}\wedge^{r+1}\mathcal{E}^{\vee}, \alpha(r+1)\otimes id,  \mathcal{F}_{2}(r+1)\otimes \pi_{S}^{*}\wedge^{r+1}\mathcal{E}^{\vee}) \to 0.
 \end{multline*}
By the additivity theorem (see Theorem 2, Collolary 3 of \cite{DQ}), the map $\sum_{i>0} (-1)^{i-1} \lambda_{i}$ is an inverse to the map $l_{n}.$ Hence the assertion.
 \end{proof}

Let $(\mathcal{F}_{1}, \alpha, \mathcal{F}_{2})\in {\bf{VB}}(f).$ Then the assignment $$u_{i}: (\mathcal{F}_{1}, \alpha, \mathcal{F}_{2})\mapsto (\pi_{S}^{*}\mathcal{F}_{1}, \pi_{X}^{*}\alpha, \pi_{S}^{*}\mathcal{F}_{2})(-i)$$ defines an exact functor from ${\bf{VB}}(f)$ to ${\bf{VB}}(\mathbb{P}(f)).$ Let $u_{i*}$ denote the induced map  $ K_{0}^{Q}(f) \to K_{0}^{Q}(\mathbb{P}(f)).$  

For notational convenience, we prefer to write $\mathcal{F}_{k}$ (resp. $\pi_{S}^{*}\mathcal{F}_{k}$) instead of $(\mathcal{F}_{k1}, \alpha_{k}, \mathcal{F}_{k2})$(resp. $(\pi_{S}^{*}\mathcal{F}_{k1}, \pi_{X}^{*}\alpha_{k}, \pi_{S}^{*}\mathcal{F}_{k2})$). We can now define a group homomorphism
$$u^{Q}: K_{0}^{Q}(f)^{r+1} \to K_{0}^{Q}(\mathbb{P}(f))$$
by sending $([\mathcal{F}_{k}])_{k=0, 1, \dots, r}$ to $\sum_{k=0}^{r} u_{k*}[\mathcal{F}_{k}]=\sum_{k=0}^{r} [u_{k}\mathcal{F}_{k}]= \sum_{k=0}^{r}[\pi_{S}^{*}\mathcal{F}_{k}(-k)].$

Note that each $u_{i}$ also induces a map $K_{0}^{He}(f) \to K_{0}^{He}(\mathbb{P}(f)).$ Therefore, in a similar way we can define a group homomorphism 
$$u^{He}: K_{0}^{He}(f)^{r+1} \to K_{0}^{He}(\mathbb{P}(f)).$$
\begin{theorem}\label{Proj formula for bass k}
 The map $u^{Q}: K_{0}^{Q}(f)^{r+1} \to K_{0}^{Q}(\mathbb{P}(f))$ is an isomorphism.
\end{theorem}

\begin{proof}
 By Theorem \ref{relative Q Resolution}, each $\mathcal{T}_{n}$ is an exact functor from ${\bf{MR}}(\mathbb{P}(f))$ to ${\bf{VB}}(f).$ Hence we can define a group homomorphism
 $$\varphi: K_{0}^{Q}{\bf{MR}}(\mathbb{P}(f)) \to K_{0}^{Q}(f)^{r+1}, [\mathcal{F}] \mapsto ([\mathcal{T}_{0}\mathcal{F}], -[\mathcal{T}_{1}\mathcal{F}], \dots, (-1)^{r}[\mathcal{T}_{r}\mathcal{F}]),$$ where $\mathcal{F}$ denotes the triples $(\mathcal{F}_{1}, \alpha, \mathcal{F}_{2}).$ Then the composition map
 $$u^{Q}\varphi:  K_{0}^{Q}(\mathbb{P}(f))\stackrel{\cong}\leftarrow K_{0}^{Q}{\bf{MR}}(\mathbb{P}(f)) \stackrel{\varphi}\to K_{0}^{Q}(f)^{r+1} \stackrel{u^{Q}}\to K_{0}^{Q}(\mathbb{P}(f))$$ sends $[\mathcal{F}]$ to $\sum_{k=0}^{r}(-1)^{k}[(\pi^{*}\mathcal{T}_{k}\mathcal{F})(-k)],$ which is equal to $[\mathcal{F}]$ by Theorem \ref{relative Q Resolution} and the additivity theorem (see Theorem 2, Collolary 3 of \cite{DQ}). This shows that $u^{Q}$ is onto.
 
 The assignment
 $$v_{i}: \mathcal{F}:=(\mathcal{F}_{1}, \alpha, \mathcal{F}_{2}) \mapsto \pi_{S*}(\mathcal{F}(i)):=(\pi_{S*}(\mathcal{F}_{1}(i)), \pi_{X*}(\alpha (i)), \pi_{S*}(\mathcal{F}_{2}(i)))$$ is also an exact functor from ${\bf{MR}}(\mathbb{P}(f))$ to ${\bf{VB}}(f)$ by Lemmas \ref{basic3} and \ref{basic2}. Let $v_{i*}$ denote the induced map $$ K_{0}^{Q}{\bf{MR}}(\mathbb{P}(f)) \to K_{0}^{Q}(f), [\mathcal{F}] \mapsto [v_{i}\mathcal{F}].$$ Using these $v_{i}$'s, we can define a group homomorphism
 $$v^{Q}: K_{0}^{Q}{\bf{MR}}(\mathbb{P}(f)) \to K_{0}^{Q}(f)^{r+1}, [\mathcal{F}] \mapsto ([v_{0}\mathcal{F}], [v_{1}\mathcal{F}], \dots, [v_{r}\mathcal{F}]).$$  
Then the composition map (using Theorem \ref{MR=VB}) $$v^{Q}u^{Q}:  K_{0}^{Q}(f)^{r+1} \to  K_{0}^{Q}(f)^{r+1}$$ is given by the matrix $(v_{i*}u_{j*}).$ Recall from Example II. 8.7.2 of \cite{wei 1} that for a quasi-coherent $\mathcal{O}_{X}$-module $\mathcal{N},$ we have $\pi_{*}\pi^{*} \mathcal{N}= \mathcal{N},$ $\pi_{*}\pi^{*} \mathcal{N}(n)= 0$ for $n<0$ and $\pi_{*}\pi^{*} \mathcal{N}(n)= {\rm {Sym}}_{n}\mathcal{E} \otimes \mathcal{N}$ for $n>0.$ Thus $$v_{i*}u_{j*}[(\mathcal{F}_{1}, \alpha, \mathcal{F}_{2})]= [(\pi_{S*}((\pi_{S}^{*}\mathcal{F}_{1})(i-j)), \pi_{X*}((\pi_{X}^{*}\alpha)(i-j)), \pi_{S*}((\pi_{S}^{*}\mathcal{F}_{2})(i-j))].$$ Since the diagram 
 
  $$\begin{CD}
     f^{*}\pi_{S*}\pi_{S}^{*}\mathcal{F}_{1}\cong \pi_{X*}\pi_{X}^{*}f^{*}\mathcal{F}_{1} @> = >> f^{*}\mathcal{F}_{1}\\
    @V \pi_{X*}\pi_{X}^{*}\alpha VV  @V \alpha VV \\                 
  f^{*}\pi_{S*}\pi_{S}^{*}\mathcal{F}_{2}\cong \pi_{X*}\pi_{X}^{*}f^{*}\mathcal{F}_{2} @> = >> f^{*}\mathcal{F}_{2}
\end{CD}$$ is commutative, $[(\pi_{S*}\pi_{S}^{*}\mathcal{F}_{1}, \pi_{X*}\pi_{X}^{*}\alpha, \pi_{S*}\pi_{S}^{*}\mathcal{F}_{2})]= [(\mathcal{F}_{1}, \alpha, \mathcal{F}_{2})]$ in $K_{0}^{Q}(f).$
 This implies that $v_{i*}u_{j*}=0$ for $i<j$ and   $v_{i*}u_{j*}= {\rm{id}}$ for $i=j.$ We get that  $ (v_{i*}u_{j*})$ is a lower triangular matrix with all of its diagonal entries equal to ${\rm{id}}.$ Therefore $v^{Q}u^{Q}$ is an isomorphism and hence $u^{Q}$ is one-one. \end{proof}
 
 Next, we prove a similar result for $K_{0}^{He}.$
 
 \begin{theorem}\label{proj formula for k heller}
  The map $u^{He}: K_{0}^{He}(f)^{r+1} \to K_{0}^{He}(\mathbb{P}(f))$ is an isomorphism.
 \end{theorem}
 
 \begin{proof}
  We consider the following commutative diagram
  $$\begin{CD}
   0 @>>> \ker(\eta^{MR}) @>>> K_{0}^{Q}({\bf{MR}}(\mathbb{P}(f)) @ > \eta^{MR} >>  K_{0}^{He}(\mathbb{P}(f)) @>>> 0 \\
   @.  @VVV           @V \cong VV             @V = VV   \\
   0 @>>> \ker(\eta^{\mathbb{P}(f)}) @>>> K_{0}^{Q}(\mathbb{P}(f)) @ >\eta^{\mathbb{P}(f)} >>  K_{0}^{He}(\mathbb{P}(f)) @>>> 0,
  \end{CD}$$
where the middle map is an isomorphism by Theorem \ref{MR=VB}. So we get $\ker(\eta^{MR})\cong \ker(\eta^{\mathbb{P}(f)}).$ Let  $\mathcal{F}_{12}:= (\mathcal{F}_{1}, \alpha, \mathcal{F}_{2}),$ $\mathcal{F}_{23}:= (\mathcal{F}_{2}, \beta, \mathcal{F}_{3})$ and $\mathcal{F}_{13}:= (\mathcal{F}_{1}, \beta\alpha, \mathcal{F}_{3})$ be in ${\bf{MR}}(\mathbb{P}(f)).$ Note that $\varphi([\mathcal{F}_{12}] + [\mathcal{F}_{23}]- [\mathcal{F}_{13}])\in \ker((\eta^{f})^{r+1})$ whenever $ ([\mathcal{F}_{12}] + [\mathcal{F}_{23}]- [\mathcal{F}_{13}]) \in \ker(\eta^{MR})$ because each $\mathcal{T}_{n}$  is an exact functor from ${\bf{MR}}(\mathbb{P}(f))$ to ${\bf{VB}}(f)$ (see Theorem \ref{Proj formula for bass k} for the map $\varphi$). Then the composition map 
$$u^{Q}\varphi:  \ker(\eta^{\mathbb{P}(f)})\stackrel{\cong}\leftarrow \ker(\eta^{MR}) \stackrel{\varphi}\to \ker((\eta^{f})^{r+1}) \stackrel{u^{Q}}\to \ker(\eta^{\mathbb{P}(f)})$$ sends $[\mathcal{F}_{12}] + [\mathcal{F}_{23}]- [\mathcal{F}_{13}]$  to $$\sum_{k=0}^{r}(-1)^{k}[(\pi_{S}^{*}\mathcal{T}_{k}\mathcal{F}_{12})(-k)] + \sum_{k=0}^{r}(-1)^{k}[(\pi_{S}^{*}\mathcal{T}_{k}\mathcal{F}_{23})(-k)] - \sum_{k=0}^{r}(-1)^{k}[(\pi_{S}^{*}\mathcal{T}_{k}\mathcal{F}_{13})(-k)]$$ which is equal to $[\mathcal{F}_{12}] + [\mathcal{F}_{23}]- [\mathcal{F}_{13}]$ by the additivity theorem  (see Theorem 2, Collolary 3 of \cite{DQ}). This shows that $u^{Q}|_{\ker((\eta^{f})^{r+1})}$ is onto. Therefore, we get the desired isomorphism from the following commutative diagram 
$$\begin{CD}
   0 @>>> \ker((\eta^{f})^{r+1}) @>>> K_{0}^{Q}(f)^{r+1} @ > (\eta^{f})^{r+1} >>  K_{0}^{He}(f)^{r+1} @>>> 0 \\
   @.  @VVV           @V u^{Q} VV             @V u^{He} VV   \\
   0 @>>> \ker(\eta^{\mathbb{P}(f)}) @>>> K_{0}^{Q}(\mathbb{P}(f)) @ >\eta^{\mathbb{P}(f)} >>  K_{0}^{He}(\mathbb{P}(f)) @>>> 0,
  \end{CD}$$ where $u^{Q}$ is an isomorphism by Theorem \ref{Proj formula for bass k} and the first vertical arrow is also an  isomorphism by the above observation. 
 \end{proof}
 
 \begin{corollary}\label{He=k}
 Suppose $\mathbb{P}(f),$ $f$ and $\mathcal{E}$ are as in Theorem \ref{proj formula for k heller}.  Further, we assume that $X$ is an affine scheme. Then there is an isomorphism of groups $K_{0}^{He}(\mathbb{P}(f))\cong K_{0}(\mathbb{P}(f)).$
 \end{corollary}
 
 \begin{proof}
 By Lemma \ref{projective bundle formula for K_n}, $K_{0}(f)^{r+1} \stackrel{\cong}\to K_{0}(\mathbb{P}(f)).$  We also have an isomorphism $K_{0}^{He}(f)^{r+1}\stackrel{\cong}\to K_{0}(f)^{r+1}$ by Lemma \ref{K^b= K for affine}. Hence we get 
  $$K_{0}(\mathbb{P}(f))\stackrel{\cong}\leftarrow K_{0}(f)^{r+1} \stackrel{\cong}\leftarrow K_{0}^{He}(f)^{r+1} \stackrel{\cong}\to K_{0}^{He}(\mathbb{P}(f)),$$ where the last isomorphism by Theorem \ref{proj formula for k heller}.  \end{proof}

\end{document}